\documentclass[reqno,11pt]{amsart}

\usepackage{mathdots}
\usepackage{tikz}
\usepackage{tikz-cd}
\usetikzlibrary{automata}
\usepackage{amssymb}
\usepackage{amsgen}
\usepackage{amsmath}
\usepackage{amsthm}
\usepackage{mathrsfs}
\usepackage{cite}
\usepackage{amsfonts}
\usepackage{enumitem}
\usepackage{stmaryrd}

\newcommand{\sour}{\mathop{\boldsymbol s}}

\newcommand{\dom}{\mathop{\mathrm{dom}}\nolimits}
\newcommand{\ran}{\mathop{\boldsymbol r}\nolimits}

\newcommand{\pd}{\mathop{\mathrm{pd}}}

\newcommand{\inv}{^{-1}}

\newcommand{\wh}{\widehat}

\newcommand{\Hom}{\mathop{\mathrm{Hom}}\nolimits}
\newcommand{\Tor}{\mathop{\mathrm{Tor}}\nolimits}
\newcommand{\Ext}{\mathop{\mathrm{Ext}}\nolimits}

\usepackage{xcolor}

\newtheorem{Thm}{Theorem}[section]
\newtheorem{Prop}[Thm]{Proposition}

\newtheorem{Lemma}[Thm]{Lemma}
{\theoremstyle{definition}
}
{\theoremstyle{remark}
\newtheorem{Rmk}[Thm]{Remark}}
\newtheorem{Cor}[Thm]{Corollary}
{\theoremstyle{remark}
}
{\theoremstyle{remark}
}

{\theoremstyle{remark}
}
{\theoremstyle{remark}
}
{\theoremstyle{remark}
}
{\theoremstyle{remark}
\newtheorem*{Claim*}{Claim}}
\newtheorem*{theorema}{Theorem A}
\newtheorem*{theoremb}{Theorem B}
\newtheorem*{theoremc}{Theorem C}

\numberwithin{equation}{section}

\title{Partial actions of free groups and groupoid homology}

\author{Benjamin Steinberg}
\address[B.~Steinberg]{%
    Department of Mathematics\\
    City College of New York\\
    Convent Avenue at 138th Street\\
    New York, New York 10031\\
    USA}
\email{bsteinberg@ccny.cuny.edu}

\thanks{The author was supported by the NSF grant DMS-2452324, a Simons Foundation Collaboration Grant, award number 849561, the Australian Research Council Grant DP230103184, and Marsden Fund Grant MFP-VUW2411.}
\date{February 16, 2026}

\keywords{Partial group actions, groupoid homology, global dimension, free groups}
\subjclass[2020]{20J05,22A22,16E10}

\begin{document}

\begin{abstract}
We give a length one projective resolution of the trivial module for the groupoid of a semi-saturated partial action (in the sense of Exel) of a free group on a compact Hausdorff and totally disconnected space.  As a consequence we obtain an elementary computation of the homology of these groupoids, which include transformation groupoids of free group actions and  Deaconu-Renault groupoids of systems $(X,T)$ where $X$ is compact Hausdorff and totally disconnected and $T$ is a local homeomorphism with domain a clopen subset of $X$.  We also show that algebra of such a partial action groupoid over a field has global dimension at most $2$ when the space is second countable.
\end{abstract}

\maketitle

\section{Introduction}
Given a collection of partial homeomorphisms $\{\theta_a\colon X_{a\inv}\to X_a\mid a\in A\}$ of a space $X$, Exel~\cite{ExelBook17} defines a partial action of the free group $F_A$ on $X$ by sending a reduced word to the corresponding composition of partial homeomorphisms; such partial actions are called semi-saturated.  Associated to a partial group action of a group $G$ on a space $X$ is an \'etale groupoid, denoted $G\ltimes X$~\cite{abadie}.  In this paper we give an explicit length one projective resolution of the trivial module $\mathbb ZX$ for $F_A\ltimes X$ when $X$ is compact Hausdorff and totally disconnected and the partial action is semi-saturated; this resolution is flat when $X$ is merely locally compact Hausdorff and totally disconnected.  As a consequence, we obtain a completely elementary computation of the homology of such a groupoid.  Such groupoids include all groupoids of Deaconu-Renault systems $(X,T)$ with $X$ compact Hausdorff and totally disconnected and with the domain of $T$ clopen by an argument of G.~de Castro.   

More precisely, our main results are the following three theorems.

\begin{theorema}\label{t:resolution}
    Let the free group $F_A$ have a semi-saturated partial action on a compact  Hausdorff totally disconnected space $X$. Let $\mathcal G=F_A\ltimes X$.    Then there is a projective resolution 
    \[0\longrightarrow \bigoplus_{a\in A}\mathbb Z\mathcal G1_{X_a}\xrightarrow{\,\partial\,} \mathbb Z\mathcal G\xrightarrow{\ran_\ast}\mathbb ZX\longrightarrow  0\]
where $\partial$ is defined on the $a$-component by $\partial(f) = f\delta_{a}-f$ where $\delta_a$ is the characteristic function of $\{a\}\times X_{a\inv}$
\end{theorema}

\begin{theoremb}\label{t:homology.of.gpd}
    Let $\theta$ be a semi-saturated partial action of a free group $F_A$ on a compact Hausdorff totally disconnected space $X$.  Let $\iota_a\colon \mathbb ZX_a\to \mathbb ZX$ be extension by $0$ and $\theta^*_a\colon \mathbb ZX_a\to \mathbb ZX$ be $f\mapsto f\circ \theta_{a}$ (extended by $0$ outside $X_{a\inv}$).  Then $H_n(F_A\ltimes X)=0$ for $n\geq 2$ and
    \begin{align*}
        H_0(F_A\ltimes X) = \mathop{\mathrm{coker}}\left(\bigoplus_{a\in A} (\iota_a-\theta_a^*)\right),\\
        H_1(F_A\ltimes X) = \ker \left(\bigoplus_{a\in A} (\iota_a-\theta_a^*)\right).
    \end{align*}
\end{theoremb}

This leads to a computation of the homology of Deaconu-Renault group\-oids without assuming the local homeomorphism is surjective or totally defined (but with restrictions on $X$).  See Corollary~\ref{c:deaconu-renault} and Remark~\ref{rmk:more.general}.

\begin{theoremc}
    Let the free group $F_A$ have a semi-saturated partial action on a second countable compact Hausdorff and totally disconnected space $X$.  Then, for any field $K$, the algebra $K[F_A\ltimes X]\cong KX\rtimes F_A$ has global dimension at most $2$.
\end{theoremc}

Theorem~C is derived from Theorem~A using a spectral sequence connecting $\Ext$ with groupoid cohomology that may be of interest in its own right.
\subsection*{Acknowledgments}
The author thanks Gilles Gon\c{c}alves de Castro for permission to use his proof of Theorem~\ref{t:castro}.

\section{Partial group actions}

\subsection{Definition of a partial action}
See~\cite{LawsonKellendonk} or~\cite{ExelBook17} for basics on partial group actions. 

Recall that an inverse semigroup is a semigroup $S$ such that, for each $s\in S$, there is a unique $s^*\in S$ such that $ss^*s=s$ and $s^*ss^*=s^*$.
If $S,T$ are inverse monoids, a dual prehomomorphism $\theta\colon S\to T$ is map  such that: 
\begin{enumerate}
    \item $\theta(1)=1$;
   \item $\theta(s^*) = \theta(s)^*$;
    \item $\theta(s)\theta(t)\leq \theta(st)$.
\end{enumerate}
  Here we recall that an inverse semigroup is partially ordered via $s\leq t$ if $ss^*t=s$ or, equivalently, $ts^*s=s$

If $X$ is a locally compact Hausdorff totally disconnected space, $I_X$ denotes the inverse semigroup of partial homeomorphisms of $X$ with clopen domains.  

A partial action of a group $G$ on $X$ is a dual prehomorphism $\theta\colon G\to I_X$.  What this entails is, for each $g\in G$, a homeomorphism $\theta_g\colon X_{g\inv}\to X_g$ between clopen subsets of $X$ such that $\theta_1=1_X$, $\theta_g\inv=\theta_{g\inv}$ and whenever $\theta_g\theta_h(x)$ is defined, it equals $\theta_{gh}(x)$.  

Let $F_A$ be a free group on $A$.  Exel~\cite{ExelBook17} calls a partial action $\theta$ of $F_A$ on $X$ \emph{semi-saturated} if $\theta_{uv}=\theta_u\circ \theta_v$ whenever $|uv|=|u|+|v|$, that is, there is no cancellation in forming the product $uv$.  Semi-saturated actions are determined uniquely by a collection of partial homeomorphisms $\theta_a\colon X_{a\inv}\to X_a$, for $a\in A$.  The action for an element with reduced form $w=a_1^{\epsilon_1}\cdots a_n^{\epsilon_n}$ must be given by $\theta_w=\theta_{a_1}^{\epsilon_1}\cdots \theta_{a_n}^{\epsilon_n}$. See~\cite[Proposition~4.10]{ExelBook17} for details.

One way to see this is to note that this collection gives us a homomorphism $\alpha\colon FIM_A\to I_X$, where $FIM_A$ is the free inverse monoid on $A$, sending $a$ to $\theta_a$.  On the other hand, there is a well-known dual prehomomorphism $\gamma\colon F_A\to FIM_A$ sending an element $w\in F_A$ to the element of $FIM_A$ represented by the reduced form of $w$, that is, the element with Munn tree~\cite{Lawson} a straight-line labeled by $w$.  Then the above partial action is $\alpha\gamma$.

Elementary properties of semi-saturated actions that we shall use throughout, often without comment, are the following.

\begin{Prop}\label{p:prefix.contain}
    Suppose that $\theta\colon F_A\to I_X$ is a semi-saturated action.  
    \begin{enumerate}
        \item  If $w=uv$ with $|w|=|u|+|v|$, then $X_w\subseteq X_u$.
        \item $\theta_u(X_{u\inv} \cap X_v)\subseteq X_{uv}$ with equality if $|uv|=|u|+|v|$.
    \end{enumerate}
\end{Prop}
\begin{proof}
    This first item is immediate as $\theta_w = \theta_u\circ \theta_v$.  For the second, if $x\in X_{u\inv}\cap X_v$, write $x=\theta_v(y)$.  Then $\theta_u(x) = \theta_u(\theta_v(y))=\theta_{uv}(y)$, and so $\theta_u(x)\in X_{uv}$.  If $|uv|=|u|+|v|$, then $\theta_{uv} = \theta_u\circ\theta_v$.   Thus, if $z\in X_{uv}$, then $z=\theta_{uv}(y) = \theta_u(\theta_v(y))$.  Therefore, $\theta_v(y)\in X_{u\inv}\cap X_v$, and so $z\in \theta_u(X_{u\inv}\cap X_v)$.
\end{proof}

Any total action of $F_A$ is, of course, a semi-saturated partial action.

Important examples of semi-saturated partial actions come from graphs.  We follow here the southern hemisphere convention on graphs: so an edge is drawn $\ran(e)\xleftarrow{\,e\,}\sour(e)$  Let $E=(E^0,E^1)$ be a graph and let $\partial E$ be the boundary path space.   For each edge $e$ put $X_e = e\partial E$ and $X_{e\inv} =\sour(e)\partial E$.  Then $\theta_e(z) = ez$.   For the corresponding semi-saturated action of $F_{E^1}$, one can check that only reduced words of the form $pq\inv$ where $p,q$ are paths with $\sour(p)=\sour(q)$ have nonempty domain. This element has domain $q\partial E$ and range $p\partial E$. 

\subsection{The groupoid of a partial group action}
A topological groupoid $\mathcal G$ is \emph{\'etale} if the source map is a local homeomorphism and the unit (or object) space $\mathcal G^0$ is locally compact Hausdorff. If, in addition,  $\mathcal G^0$ is totally disconnected, the groupoid is said to be \emph{ample}.   Recall that a bisection of $\mathcal{ G}$ is an open subset $U$ on which the source and range maps are both injective.  The compact open bisections of an ample groupoid form a basis for the topology and an inverse semigroup under multiplication.

Abadie~\cite{abadie} associated a Hausdorff \'etale groupoid $G\ltimes X$ to a partial group action of $G$ on $X$.
The object space is $X$ and the arrow space is $\{(g,x)\mid x\in X_{g\inv}\}\subseteq G\times X$ with the subspace topology coming from the product topology.  One has $\sour(g,x)=x$, $\ran(g,x)=\theta_g(x)$, $(g,\theta_h(x))(h,x) = (gh,x)$ and $(g,x)\inv = (g\inv,\theta_g(x))$.   The identities are $(1,x)$ with $x\in X$.

If $E$ is a graph, and we consider the semi-saturated partial action of $F_{E^1}$ on $\partial E$, then $F_{E^1}\ltimes \partial E$ is well-known to be isomorphic to the usual graph groupoid $\mathcal G_E$ via  $(pq\inv, qz)\mapsto (pz, |p|-|q|, qz)$.  

More generally, the groupoid of any Deaconu-Renault system $(X,T)$ with $X$ compact Hausdorff and totally disconnected and with the domain of $T$ clopen is the groupoid of a semi-saturated partial action of  a free group.   Recall that a Deaconu-Renault system consists of a locally compact Hausdorff space $X$  and a local homeomorphism $T\colon \dom(T)\to X$ where $\dom(T)$ is an open subset of $X$.   The Deaconu-Renault groupoid $\mathcal G_{(X,T)}$ has object space $X$ and arrows \[\{(x,m-n,y)\in \dom(T^m)\times \mathbb Z\times \dom(T^n)\mid T^m(x)=T^n(y)\}\] with $\sour(x,k,y)=y$, $\ran(x,k,y)=x$. The topology has basis the bisections of the form \[(U,(m,n),V) = \{(x,m-n,y)\in U\times \mathbb Z\times V\mid T^m(x)=T^n(y)\}\] with $U\subseteq \dom(T^m)$, $V\subseteq \dom(T^n)$ open, $T|_U,T|_V$ injective and $T^m(U)=T^n(V)$.   The graph groupoid is the Deaconu-Renault system of the shift map on the boundary path space $\partial E$.

A partial action of a free group $F_A$ on a space $X$ is called \emph{orthogonal} if $X_a\cap X_b=\emptyset$ for $a\neq b\in A$~\cite{subshift.semigroups}.
The argument below was given to me by Gilles de Castro (private communication).
\begin{Thm}\label{t:castro}
    The groupoids of Deaconu-Renault systems $(X,T)$ with $X$ compact Hausdorff and totally disconnected and $\dom(T)$ clopen are precisely the groupoids of semi-saturated orthogonal partial actions of finitely generated free groups on compact Hausdorff and totally disconnected spaces with clopen domains.
\end{Thm}
\begin{proof}
    If $A$ is finite and $\theta$ is an orthogonal and semi-saturated partial action of $F_A$ on a compact Hausdorff and totally disconnected $X$, then we can define a local homeomorphism with clopen domain $T\colon \bigcup_{a\in A} X_a\to X$ by $T|_{X_a} = \theta_{a\inv}$, using orthogonality of the action.  Then~\cite[Theorem~9.6]{subshift.semigroups}  shows that $F_A\ltimes X\cong \mathcal G_{(X,T)}$ (see also~\cite[Theorem~4.1]{semisatcastro}).   
    
    Conversely, if $(X,T)$ is a Deaconu-Renault system with with the domain $\dom(T)$ of $T$ clopen (hence compact), then we can partition $\dom(T)$ into finitely many clopen sets $U_1\sqcup\cdots \sqcup U_n$ with $T|_{U_i}$ a homeomorphism with its image $V_i$.  We can then define a semi-saturated orthogonal partial action of the free group on $a_1,\ldots, a_n$ on $X$ by putting $X_{a_i}=U_i$, $X_{a_i\inv} = V_i$ and $\theta_{a_i} = (T|_{U_i})\inv$ using~\cite[Proposition~4.10]{ExelBook17}.  Note then that $T=\bigcup_{i=1}^n\theta_{a_i\inv}$, and so $F_A\ltimes X\cong \mathcal G_{(X,T)}$ by the previous paragraph.
\end{proof}

We now show that groupoids  $F_A\ltimes X$ of semi-saturated partial actions are exactly the groupoids of germs $FIM_A\ltimes X$ for actions of the free inverse monoid $FIM_A$ on $A$.  See~\cite{Lawson} for more on the free inverse monoid and~\cite{Exel} for groupoids of germs.

\begin{Prop}\label{p:fim}
    Let $\{\theta_a\colon X_{a\inv}\to X_A\colon a\in A\}$ be a collection of homeomorphisms of a locally compact Hausdorff totally disconnected space $X$ with the $X_a$ clopen.  Then $F_A\ltimes X\cong FIM_A\ltimes X$, where the groupoid of germs $FIM_A\ltimes X$ is taken with respect to the action of $FIM_A$ on $X$ given by  $a\mapsto \theta_a$.
\end{Prop}
\begin{proof}
    The key fact that we use is that $FIM_A$ is an $F$-inverse monoid.  More precisely, if $\sigma\colon FIM_A\to F_A$ is the canonical projection and if $w\in F_A$ is viewed as a reduced word over $A$, then $w$ evaluated in $FIM_A$ is the maximum element of $\sigma\inv(w)$.   Thus each germ of $FIM_A\ltimes X$ is uniquely represented by a germ $[w,x]$ with $w$ reduced since if $[u,x]=[v,x]$ then $\sigma(u)=\sigma(v)$. 
    Moreover, basic neighborhoods in $FIM_A\ltimes X$ can be taken to be of the form $D(w,U) =\{[w,u]\mid u\in U\}$  where $w$ is a reduced word and $U$ is an open subset of $X_{w\inv}$.  The actions of a reduced word viewed as an element of $F_A$ and as an element of $FIM_A$ are the same.  
    It follows directly that $(w,x)\mapsto [w,x]$ is an isomorphism of topological groupoids $F_A\ltimes X\to FIM_A\ltimes X$.
\end{proof}

\section{Groupoid algebras and homology}
\subsection{Groupoid algebras}
If $K$ is a commutative ring and $X$ is a locally compact Hausdorff and totally disconnected space, then $KX$ will denote the $K$-module of locally constant functions $f\colon X\to K$ with compact support.   
If $\mathcal G$ is a Hausdorff ample groupoid, then $K\mathcal G$ becomes a $K$-algebra~\cite{mygroupoidalgebra} with respect to the convolution product
\[f\ast g(\gamma) = \sum_{\ran(\alpha)=\ran(\gamma)}f(\alpha)g(\alpha\inv \gamma).\]
Often we write the convolution product as juxtaposition. 

In the case of $G\ltimes X$, one has that $K[G\ltimes X]$ is isomorphic to the partial skew group ring $KX\rtimes G$ where $KX$ is given the pointwise product; see~\cite{ExelDoku} for more on partial skew group rings.

To simplify the discussion, we assume from now on that $X$ is compact, and hence each $X_g$ is compact.   Let $\delta_g$ be the characteristic function of the compact open bisection $\{g\}\times X_{g\inv}$.   Then the clopen partition $G\ltimes X=\bigsqcup_{g\in G} \{g\}\times X_{g\inv}$ leads to a $K$-module direct sum decomposition \[K[G\ltimes X] = \bigoplus KX_g\delta_g.\]
That is, if $f$ is supported on $\{g\}\times X_{g\inv}$, we can write $f= (f\delta_{g\inv})\delta_g$ and $f\delta_{g\inv}$ is in $KX_g$.  In particular, in the direct sum decomposition $\delta_g = 1_{X_g}\delta_g$.

The reader can check that if $\alpha_g\colon KX_{g\inv}\to KX_g$ is defined by $\alpha_g(f) = f\circ \theta_{g\inv}$, then 
\[f\delta_g\ast f'\delta_g' = \alpha_g(\alpha_{g\inv}(f)f')\delta_{gg'} \] although we shall not use this formula much.

\begin{Prop}\label{p:bisection.prod}
Let $F_A$ have a semi-saturated partial action on $X$.
Let $u,v\in F_A$.  Then
\[(\{u\}\times X_{u\inv})(\{v\}\times X_{v\inv})\subseteq \{uv\}\times X_{(uv)\inv}\] with equality if $|uv|=|u|+|v|$.  Consequently, if $|uv|=|u|+|v|$, then $\delta_u\delta_v = \delta_{uv}$.    
\end{Prop}
\begin{proof}
    The inclusion is immediate from the definition of a partial action.  If $|u|+|v|=|uv|$ and $x\in X_{(uv)\inv}$, then $\theta_{uv}(x) = \theta_u(\theta_v(x))$ by the semi-saturated property, and so $(uv,x) = (u,\theta_v(x))(v,x)\in (\{u\}\times X_{u\inv})(\{v\}\times X_{v\inv})$.
\end{proof}

\subsection{Homology and cohomology of ample groupoids}
Let $\mathcal G$ be a Hausdorff ample groupoid.  Then $K\mathcal G^0$ is a left $K\mathcal G$-module via
\[f\cdot g(x) = \sum_{\ran(\gamma)=x}f(\gamma)g(\sour(\gamma)).\] We call this the trivial left $K\mathcal G$-module. If $U$ is a compact open bisection, then $1_U\cdot g=1_U\ast g\ast 1_{U\inv}$.  The action restricts to $K\mathcal G_0$ as left miltiplication.   The trivial right module is defined dually.

The \emph{homology} of $\mathcal G$ with coefficients in a right $\mathbb Z\mathcal G$-module $M$ is \[H_n(\mathcal G,M) = \Tor_n^{\mathbb Z\mathcal G}(M,\mathbb Z\mathcal G^0)\] and the \emph{cohomology} of $\mathcal G$ in a left $\mathbb Z\mathcal G$-module $M$ is given by \[H^n(\mathcal G,M) = \Ext^n_{\mathbb Z\mathcal G}(\mathbb Z\mathcal G^0,M).\]

One puts $H_n(\mathcal G) = H_n(\mathcal G,\mathbb Z\mathcal G^0)$  and $H^n(\mathcal G) = H^n(\mathcal G,\mathbb Z\mathcal G^0)$.
A good reference for groupoid homology is~\cite{BDGW}.

The \emph{cohomological dimension} of $\mathcal G$ is the shortest length of a projective resolution of $\mathbb Z\mathcal G^0$, or, equivalently, the least $n$ such that $H^{n+1}(\mathcal G,M)=0$ for all $\mathbb Z\mathcal G$-modules $M$.

The famous Stallings-Swan theorem says that free groups are precisely the groups of cohomological dimension $1$.   We show that semi-saturated actions of free groups on compact Hausdorff totally disconnected spaces lead to groupoids of cohomological dimension one.

If $f\colon X\to Y$ is a local homeomorphism of locally compact totally disconnected Hausdorff spaces, we have an induced homomorphism $f_\ast\colon \mathbb ZX\to \mathbb ZY$ given by $f_*(g)(y) = \sum_{f(x)=y}g(x)$.  

It is well known that $\ran_\ast\colon \mathbb Z\mathcal G\to \mathbb Z\mathcal G^0$ is a $\mathbb Z\mathcal G$-module homomorphism.  If $U$ is a compact open bisection, then $\ran_\ast(1_U) = 1_{\ran(U)}$.

\section{Cohomological dimension and semi-saturated partial actions}
For this section we fix a semi-saturated partial action $\theta$ of a free group $F_A$ on a compact totally disconnected Hausdorff space $X$.  We put $\mathcal G=F_A\ltimes X$.  Our goal is to give a length one projective resolution of $\mathbb Z \mathcal G^0=\mathbb ZX$.   

We begin our projective resolution with $\ran_\ast\colon \mathbb Z\mathcal G\to \mathbb ZX$.  It will be convenient to understand this map with respect to our direct sum decomposition $\mathbb Z\mathcal G = \bigoplus_{w\in F_A}\mathbb ZX_w\delta_w$.

\begin{Prop}\label{p:r*.nice}
    Suppose that $f\in \mathbb ZX_w$.  Then $\ran_\ast(f\delta_w)=f$.
\end{Prop}
\begin{proof}
Indeed,    $\ran_\ast(f\delta_w) = f\ran_\ast(\delta_w) = f1_{X_w}= f$ using that $\ran_\ast$ is a module homomorphism and $\delta_w = 1_{\{w\}\times X_{w\inv}}$.  
\end{proof}

Set $P_0=\mathbb Z\mathcal G$.  
Our next projective module will be $P_1=\bigoplus_{a\in A}\mathbb Z\mathcal G1_{X_a}\gamma_a$ where $\gamma_a$ is just a symbol to indicate to which summand an element belongs.

\begin{Prop}\label{p:belongs}
 One has $\mathbb ZX_w\delta_w1_{X_a} = \mathbb Z[\theta_w(X_{w\inv} \cap X_a)]\delta_w$.  Hence $\mathbb Z\mathcal G1_{X_a}=\bigoplus_{w\in F_X}\mathbb Z [\theta_w(X_w\inv \cap X_a)]\delta_w$.
\end{Prop}
\begin{proof}
Note that $(\{w\}\times X_{w\inv})(\{1\}\times X_a)$ consists of all $(w,x)$ with $x\in X_{w\inv}\cap X_a$, as does $(\{1\}\times \theta_w(X_{w\inv}\cap X_a))(\{w\}\times X_{w\inv})$.  Thus $\delta_w 1_{X_a} = 1_{\theta_w(X_{w\inv}\cap X_a)}\delta_w$, from which the proposition follows.
\end{proof}

We think of $f\delta_w\gamma_a$ as the edge $w\xrightarrow{a}wa$ in the Cayley graph of $F_A$ with weight $f$.
With this intuition, we define $\partial\colon P_1\to P_0$ by 
\[\partial (f\delta_w\gamma_a) = f\delta_{wa}-f\delta_w.\]
Note that Proposition~\ref{p:belongs} and Proposition~\ref{p:prefix.contain}(2) imply that $f\in \mathbb ZX_{wa}$.  Therefore, $\ran_\ast\partial(f\delta_w\gamma_a) = f-f=0$ by Proposition~\ref{p:r*.nice}. 

\begin{Prop}\label{p:partial.nice}
For $f\in \mathbb Z\mathcal G1_{X_a}$, one has $\partial(f\gamma_a)=f\delta_a-f$. Consequently, $\partial$ is a $\mathbb Z\mathcal G$-module homomorphism.  
\end{Prop}
\begin{proof}
One computes $\delta_w\delta_a = 1_{\theta_w(X_{w\inv}\cap X_a)}\delta_{wa}$.  So, if 
 $h\delta_w\gamma_a\in \mathbb Z\mathcal G1_{X_a}$, then $h\delta_{wa} = h1_{\theta_w(X_{w\inv}\cap X_a)}\delta_{wa}=h\delta_w\delta_a$ as a consequence of Proposition~\ref{p:belongs}, and so $\partial(h\delta_w\gamma_a) = h\delta_{wa}-h\delta_w=h\delta_w\delta_a-h\delta_w$, as required.
\end{proof}

We now show that the chain complex $0\to P_2\xrightarrow{\partial}P_1\xrightarrow{\ran_\ast}\mathbb ZX\to 0$ is acyclic by constructing a contracting homotopy over $\mathbb Z$, based on the corresponding contracting homotopy for the special case that $X$ is a singleton, i.e., for the projective resolution of the trivial module for the free group
\[0\to \bigoplus_{a\in A} \mathbb ZF_A\gamma_a\to \mathbb ZF_A\to \mathbb Z\to 0.\] In that contracting homotopy, $s_1\colon \mathbb ZF_A\to \bigoplus_{a\in A}\mathbb ZF_A\gamma_a$ sends a word $w$  to the signed sum of the edges in the Cayley graph of $F_A$ on the geodesic path from $1$ to $w$ where we view $g\gamma_a$ as the edge $g\xrightarrow{a}ga$.   We shall emulate this contracting homotopy in our context.  The semi-saturated condition will ensure that if $f\in X_w$, then $f\in X_u$ for any prefix $u$ of $w$ by Proposition~\ref{p:prefix.contain}, which will allow us to perform this homotopy.

First define $s_0\colon \mathbb ZX\to \mathbb Z\mathcal G=P_0$ by $s_0(f) = f\delta_1$.   The definition of $s_1\colon P_0\to P_1$ is more involved, as we must encode signed traversals of edges.

Let $a\in A$ and $w\in F_A$.
If $f\in X_{wa}$ and $|wa|=|w|+1$, let us put \[f[w,a] = f\delta_{w}\gamma_a.\]  If $f\in X_{wa\inv}$ with $|wa\inv| = |w|+1$, then we put \[f[w,a\inv] = -f\delta_{wa\inv}\gamma_a.\]
(Imagine traversing the edge $wa\inv\xrightarrow{a}w$ in the negative direction.)

\begin{Lemma}\label{l:homotopy.stuff}
Suppose that $|wa^{\epsilon}| =|w|+1$ for $a\in A$, $\epsilon=\pm 1$ and $f\in X_{wa^\epsilon}$.
\begin{enumerate}
    \item $f[w,a^\epsilon]\in \mathbb Z\mathcal G1_{X_a}\gamma_a$.
    \item $\partial (f[w,a^\epsilon]) = f\delta_{wa^\epsilon} - f\delta_w$.
\end{enumerate}
\end{Lemma}
\begin{proof}
We first prove the two items when $\epsilon=1$.  Proposition~\ref{p:prefix.contain}(2) implies that $f\in \mathbb Z[\theta_w(X_{w\inv}\cap X_a)]$.  Moreover, $\partial (f[w,a]) = \partial (f\delta_w\gamma_a) = f\delta_{wa}-f\delta_w$.

Suppose now that $\epsilon =-1$.   Then $f\delta_{wa\inv} =f\delta_w\delta_{a\inv} = f\delta_w\delta_{a\inv}1_{X_a}\in \mathbb Z\mathcal G1_{X_a}$ where the first equality is from Proposition~\ref{p:bisection.prod}.  Furthermore, \[\partial (f[w,a\inv]) =- \partial (f\delta_{wa\inv}\gamma_a) = -(f\delta_{w}-f\delta_{wa\inv}) = f\delta_{wa\inv} - f\delta_w,\] as required.
\end{proof}

We now define $s_1\colon P_0\to P_1$ as follows.  Suppose that $w=a_1^{\epsilon_1}\cdots a_n^{\epsilon_n}$ in reduced form with $a_i\in A$, $\epsilon_i\in \{\pm 1\}$ and $n\geq 0$. Set, for $f\in \mathbb ZX_w$,
\[s_1(f\delta_w) = \sum_{i=1}^n f[a_1^{\epsilon_1}\cdots a_{i-1}^{\epsilon_{i-1}},a_i^{\epsilon_i}],\]
where we note that by Proposition~\ref{p:prefix.contain}(1), we have that $f\in \mathbb ZX_{a_1^{\epsilon_1}\cdots a_i^{\epsilon_i}}$ for all $0\leq i\leq n$ since $f\in \mathbb ZX_w$.  Observe that if $w=1$, then $s_1(f\delta_1)=0$ since the sum is empty.

\begin{Prop}\label{p:contracting}
    The maps $s_0,s_1$ form a contracting homotopy of the chain complex of abelian groups $0\to P_1\to P_0\to \mathbb ZX\to 0$.
\end{Prop}
\begin{proof}
Clearly, $s_0,s_1$ are homomorphisms of abelian groups. 
To show that they form a contracting homotopy of abelian groups, we need to verify: $\ran_\ast s_0=1$; $\partial s_1+s_0\ran_\ast=1$;
    and $s_1\partial = 1$.  We begin with the first property.
    
    If $f\in \mathbb ZX$, then $\ran_\ast s_0(f) = \ran_\ast(f\delta_1) = f$ by Proposition~\ref{p:r*.nice}.

Next we turn to the second property: $\partial s_1+s_0\ran_\ast=1$.
    We apply the left hand side to $f\delta_w$ with $f\in \mathbb ZX_w$.  Suppose that $w=a_1^{\epsilon_1}\cdots a_n^{\epsilon_n}$ in reduced form with $n\geq 0$.  Note that $s_0\ran_\ast (f\delta_w) = f\delta_1$ by Proposition~\ref{p:r*.nice}.   Thus,
    \begin{align*}
        (\partial s_1+s_0\ran_\ast)(f\delta_w) &= \sum_{i=1}^n \partial f[a_1^{\epsilon_1}\cdots a_{i-1}^{\epsilon_{i-1}},a_i^{\epsilon_i}]+f\delta_1\\
        &= \sum_{i=1}^n (f\delta_{a_1^{\epsilon_1}\cdots a_{i}^{\epsilon_{i}}}- f\delta_{a_1^{\epsilon_1}\cdots a_{i-1}^{\epsilon_{i-1}}})+f\delta_1\\
        &= f\delta_w 
    \end{align*}
since the sum is telescoping, where we have used Lemma~\ref{l:homotopy.stuff}(2)
for the second equality.

Finally, we prove the third property:
    the equality $s_1\partial =1$ holds.
    We compute $s_1\partial(f\delta_w\gamma_a)$.    Let $w=a_1^{\epsilon_1}\cdots a_n^{\epsilon_n}$ in reduced form with $n\geq 0$.   There are two cases. If $a_n^{\epsilon_1}\neq a\inv$, then $|wa|=|w|+1$ and we compute
    \begin{align*}
        s_1\partial(f\delta_w\gamma_a)&=s_1(f\delta_{wa}-f\delta_w) \\
        &=  \sum_{i=1}^n f[a_1^{\epsilon_1}\cdots a_{i-1}^{\epsilon_{i-1}},a_i^{\epsilon_i}]+f[w,a]- \sum_{i=1}^n f[a_1^{\epsilon_1}\cdots a_{i-1}^{\epsilon_{i-1}},a_i^{\epsilon_i}]\\
        &= f\delta_w\gamma_a
    \end{align*}
    If $a_n^{\epsilon_n}=a\inv$, put $u=a_1^{\epsilon_1}\cdots a_{n-1}^{\epsilon_{n-1}}$, so that $w=ua\inv$.
    Then we have 
    \begin{align*}
        s_1\partial(f\delta_w\gamma_a)&=s_1(f\delta_{u}-f\delta_{ua\inv}) \\
        &=  \sum_{i=1}^{n-1} f[a_1^{\epsilon_1}\cdots a_{i-1}^{\epsilon_{i-1}},a_i^{\epsilon_i}]- \sum_{i=1}^{n-1} f[a_1^{\epsilon_1}\cdots a_{i-1}^{\epsilon_{i-1}},a_i^{\epsilon_i}]-f[u,a\inv]\\
        &= f\delta_{ua\inv}\gamma_a=f\delta_w\gamma_a,
    \end{align*}
    as required.
\end{proof}

\begin{proof}[Proof of Theorem~A]
The exactness of the resolution follows from the existence of a contracting homotopy over $\mathbb Z$; see Proposition~\ref{p:contracting}.  Since $X$ is compact, $\mathbb Z\mathcal G$ is unital, and hence projective as a $\mathbb Z\mathcal G$-module. Since $1_{X_a}$ is an idempotent, each $\mathbb Z\mathcal G1_{X_a}$ is projective.  This completes the proof.
\end{proof}


\begin{Cor}\label{c:dim.one}
Let $F_A$ have a semi-saturated partial action on a  compact totally disconnected Hausdorff space $X$. Then $\mathcal G=F_A\ltimes X$ has cohomological dimension at most $1$.  Hence for any $\mathbb Z\mathcal G$-module $M$, $H_n(\mathcal G,M)=0$ and $H^n(\mathcal G,M)=0$ for $n\geq 2$.  
\end{Cor}

\begin{Cor}
Let $(X,T)$ be a Deaconu-Renault system with $X$ compact Hausdorff and totally disconnected and the domain of $T$ clopen.  Then $\mathcal G_{(X,T)}$ has cohomological dimension at most $1$.
\end{Cor}

We are now in a position to prove Theorem~B.

\begin{proof}[Proof of Theorem~B]
    Tensoring the deleted projective resolution from Theorem~A with the right module $\mathbb ZX$ yields the chain complex
    \[\bigoplus_{a\in A}\mathbb ZX_a\xrightarrow{\,d\,}\mathbb ZX\] computing homology where $d$ is identified with $1_{\mathbb ZX}\otimes \partial$ under the isomorphisms $\mathbb ZX\otimes_{\mathbb Z\mathcal G}\mathbb Z\mathcal{G}1_{X_A}\cong \mathbb ZX_a$ and $\mathbb ZX\otimes_{\mathbb Z\mathcal G}\mathbb Z\mathcal{G}\cong \mathbb ZX$.  If $f\in \mathbb ZX_a$, then $d(f)$ corresponds to $(1_{\mathbb ZX}\otimes \partial)(1_{X}\otimes f) = 1_{X}\otimes \partial f = 1_{X}\otimes f\delta_a-1_{X}\otimes f = \delta_{a\inv}f\delta_a\otimes 1_{X_{a\inv}} -f\otimes 1_{X_a}= f\circ \theta_{a}\otimes 1_{X_{a\inv}}-f\otimes 1_{X_a}$, which corresponds to $\theta_a^*(f)-\iota_a(f)$ under the isomorphisms.  The result now follows by definition of  groupoid homology as $\mathrm{Tor}_\bullet^{\mathbb Z\mathcal G}(\mathbb Z\mathcal{G}^0,\mathbb Z\mathcal{G}^0)$.
\end{proof}

Theorem~B reduces to a known result for Deaconu-Renault group\-oids \cite[Theorem~6.7]{fkps19}, although that reference assumes that the local homeomorphism is fully defined and surjective, whereas we allow it to be defined on a clopen subset and not be surjective.  However, we require $X$ to be compact, although that can be relaxed somewhat; see Remark~\ref{rmk:more.general}.  

\begin{Cor}\label{c:deaconu-renault}
    Let $(X,T)$ be a Deaconu-Renault groupoid with $X$ compact Hausdorff and totally disconnected and where $T$ has clopen domain $U$.  Then $H_n(\mathcal G_{(X,T)})=0$ for $n\geq 2$ and
    \begin{align*}
        H_0(\mathcal G_{(X,T)}) &= \mathop{\mathrm{coker}}(\iota-T_*)\\
        H_1(\mathcal G_{(X,T)}) &= \ker (\iota- T_*).
    \end{align*}
    where $\iota\colon \mathbb ZU\to \mathbb ZX$ is extension by $0$ and $T_\ast\colon \mathbb ZU\to\mathbb ZX$ is the induced homomorphism.
\end{Cor}
\begin{proof}
    The proof of Theorem~\ref{t:castro} shows that there is an orthogonal semi-saturated partial action $\theta$ of a free group $F_A$ on a finite set $A$ such that $U = \bigsqcup_{a\in A} X_a$, $T=\bigcup_{a\in A}\theta_{a\inv}$ and $\mathcal G_{(X,T)} \cong F_A\ltimes X$.  We can identify $\bigoplus_{a\in A}\mathbb ZX_a$ with $\mathbb ZU$ via the inclusions, and under this identification, using $T=\bigcup_{a\in A}\theta_{a\inv}$, the map $\bigoplus_{a\in A}(\iota_a-\theta_a^*)$ becomes $\iota-T_*$.  Indeed, $T_*(f)(x)=\sum_{T(y)=x}f(y) = \sum_{a\in A}f(\theta_a(x))=\sum_{a\in A}\theta_a^*(f)(x)$.  The result follows from Theorem~B.
\end{proof}

Notice that the isotropy groups of $F_A\ltimes X$ are subgroups of $F_A$, and hence free.  Thus they are torsion-free and satisfy the Baum-Connes conjecture.  Since $H_n(F_A\ltimes X)=0$ for $n\geq 2$, the Matui HK conjecture holds for $F_A\ltimes X$ and we have $K_i(C^*_r(F_A\ltimes X))\cong H_i(F_A\ltimes X)$ for $i=0,1$ by~\cite[Remark~3.5]{PY22}.  By Abadie's results~\cite{abadie}, $C^*_r(F_A\ltimes X)$ is the reduced partial crossed product $C(X)\rtimes_r F_A$.  

\begin{Rmk}\label{rmk:more.general}
     Most of the results of this paper can be adapted to the case where $X$ is merely assumed to be locally compact Hausdorff and totally disconnected, except that one will obtain a flat resolution instead of a projective resolution in Theorem~A (as $\mathbb Z\mathcal G$ and its direct summands are merely flat), and so one can only make conclusions about homology and not cohomology.   In this context, $\delta_g$ is just the characteristic function of a clopen bisection, but convolving it with a compactly supported locally constant function still gives a compactly supported  locally constant function, and so our products with it make sense when adapting the arguments.  The proof that Deaconu-Renault groupoids come from orthogonal semi-saturated partial actions of free groups in the locally compact Hausdorff and totally disconnected case requires $\dom T$ to be paracompact in order to partition it into disjoint compact open sets on which $T$ is injective.   Otherwise, there are no obstacles and, in particular, Theorem~B and Corollary~\ref{c:deaconu-renault} still go through under these hypotheses.
\end{Rmk}

\section{Global dimension}
In this section, we show that for semi-saturated partial group actions of a free group $F_A$, the algebra $K[F_A\ltimes X]$ has global dimension at most $2$ if $K$ is a field and $X$ is compact Hausdorff totally disconnected and second countable. I conjecture such algebras are, in fact, hereditary.  This would generalize the well-known fact that Leavitt path algebras of finite graphs are hereditary~\cite{LeavittBook} and free group algebras are hereditary.

Recall that a ring is (left) \emph{hereditary} if each left ideal is projective.  
It is well known that a ring is  hereditary if and only if any submodule of a projective module is projective. 

The (left) \emph{global dimension} of a ring $R$ is the supremum of the projective dimensions of its (left) modules.  A ring is hereditary if and only if it has global dimension at most one.   A result of Auslander~\cite[Theorem~8.16]{Rotmanhom} says that the global dimension of $R$ is the supremum of the projective dimensions of cyclic modules $R/L$ with $L$ a left ideal.  In particular, by considering the exact sequence $0\to L\to R\to R/L\to 0$, it follows that if each left ideal has projective dimension at most $n$, then $R$ has global dimension at most $n+1$.  

We develop some preliminaries.   First, we assume that $\mathcal G$ is a Hausdorff ample groupoid with compact unit space in what follows.   If $X$ is a locally compact and totally disconnected Hausdorff space and $K$ is a commutative ring, then $K\otimes_{\mathbb Z} \mathbb ZX\cong KX$ by~\cite[Corollary~2.3]{Li_2025}. 
We use $\pd_R M$ for the projective dimension of a module over a ring $R$.

\begin{Lemma}\label{l:pd.drops}
    Let $K$ be a commutative ring.  Then $\pd_{K\mathcal G}K\mathcal G^0\leq \pd_{\mathbb Z\mathcal G}\mathbb Z\mathcal G^0$.
\end{Lemma}
\begin{proof}
    First note that since $K\mathcal G\cong K\otimes_{\mathbb Z}\mathbb Z\mathcal G$, extension of scalars preserves projectivity.  Let $P_\bullet\to \mathbb Z\mathcal G^0$ be a projective resolution of length $\pd_{\mathbb Z\mathcal G}\mathbb Z\mathcal G^0$.  Then $K\otimes_{\mathbb Z}P_\bullet\to K\mathcal G^0$ is a chain complex of projective $K\mathcal G$-modules.  Since $\mathbb Z\mathcal G$ is torsion-free as an abelian group,  $P_\bullet\to \mathbb Z\mathcal G^0$ is a flat resolution over $\mathbb Z$ of the torsion-free, hence flat, abelian group $\mathbb Z\mathcal G^0$.  Thus $H_n(K\otimes_{\mathbb Z}P_\bullet)\cong \Tor_n^{\mathbb Z}(K,\mathbb Z\mathcal G^0)=0$.   We conclude that $K\otimes_{\mathbb Z}P_\bullet\to K\mathcal G^0$ is a projective resolution of length $\pd_{\mathbb Z\mathcal G}\mathbb Z\mathcal G^0$, establishing the result. 
\end{proof}

Let $S$ be the inverse monoid of compact open bisections of $\mathcal G$.   Then $K\mathcal G\cong KS/I$ via $1_U\mapsfrom U$ where $I$ is generated by all differences $U+V - (U\cup V)$ where $U,V\subseteq \mathcal G^0$ are disjoint and compact open~\cite[Theorem~3.11]{mygroupoidarxiv}. Moreover, any idempotent of $S$, such as $ss^*$, maps into $K\mathcal{G}^0$

Let $M,N$ be left $K\mathcal G$-modules.  We wish to show that $M\otimes_{K\mathcal G^0} N$ is a $K\mathcal G$-module.  First note that $M\otimes_K N$ is a $KS$-module in the standard way: $s(m\otimes n) = sm\otimes sn$.  Note that $M\otimes_{K\mathcal G^0}N$ is the quotient of $M\otimes_K N$ by the $K$-module $A$ generated by the elements $em\otimes n -m\otimes en$ with $e=1_U$ where $U\subseteq \mathcal G^0$ compact open.  If $s=1_V$ with $V$ a compact open bisection, then $ses^* = 1_{VUV\inv}$.  Therefore, $s(em\otimes n -m\otimes en) = ses^*s\otimes sn-sm\otimes ses^*sn$ belongs to $A$.  It follows that $M\otimes_{K\mathcal G^0} N$ is a $KS$-module.  To show that it is a $K\mathcal G$-module we must show that if $m\otimes n\in M\otimes_{K\mathcal G^0} N$ is a basic tensor and $U,V\subseteq G^0$ are disjoint compact open sets, then $(1_U+1_V)(m\otimes n) = 1_{U\cup V}(m\otimes n)$.  Indeed, we have that
\begin{align*}
1_{U\cup V}(m\otimes n) &=1_{U\cup V}m\otimes 1_{U\cup V}n 
=m\otimes 1_{U\cup V}n
=m\otimes (1_U+1_V)n \\
&= m\otimes 1_Un+m\otimes 1_Vn= 1_Um\otimes 1_Un+1_Vm\otimes 1_Vn\\ &= 1_U(m\otimes n)+1_V(m\otimes n)= (1_U+1_V)(m\otimes n)
\end{align*}
 since the tensor product is over $K\mathcal G^0$.

It is trivial to verify that the natural isomorphism $(L\otimes_{K\mathcal G^0} M)\otimes_{K\mathcal G^0} N\cong L\otimes_{K\mathcal G^0} (M\otimes_{K\mathcal G^0} N)$ is a $K\mathcal G$-module isomorphism.

Notice that the action of $S$ on the trivial module $K\mathcal G^0$ is $(s,f)\mapsto sfs^*$.

\begin{Prop}\label{p:unit}
There is a natural isomorphism $KG^0\otimes_{K\mathcal G^0} M\to M$ for any $K\mathcal G$-module $M$.
\end{Prop}
\begin{proof}
The action map induces a $K\mathcal G^0$-module isomorphism $f\otimes m\mapsto fm$.   We check this is a $K\mathcal G$-module isomorphism.  Indeed, if $s\in S$ and $e=1_U$ with $U\subseteq \mathcal G^0$ compact open, then $s(e\otimes m) = ses^*\otimes sm\mapsto ses^*sm=ss^*sem=sem$.
\end{proof}

We can similarly make $\Hom_{K\mathcal G^0}(M,N)$ into a $K\mathcal G$-module for $K\mathcal G$-modules $M,N$. The action is $(sf)(m) = sf(s^*m)$. To see that  this is $K\mathcal G^0$-linear, let $e=1_U$ with $U\subseteq \mathcal G^0$ compact open and let $s\in S$.  Then $(sf)(em) = sf(s^*em) = sf(s^*ess^*m) = ss^*esf(s^*m) = esf(s^*m)=e(sf)(m)$.  This makes $\Hom_{K\mathcal G^0}(M,N)$ into a $KS$-module.  To check that it is a $K\mathcal G$-module, we take disjoint compact open  $U,V\subseteq \mathcal G^0$ and $f\colon M\to N$, a $K\mathcal G^0$-module homomorphism, and check that $1_{U\cup V}f = 1_Uf+ 1_Vf$.  Indeed,
\begin{align*}
    (1_{U\cup V}f)(m) &= 1_{U\cup V}f(1_{U\cup V}m) = f(1_{U\cup V}m)=f((1_U+1_V)m) \\ &= f(1_Um)+f(1_Vm) = 1_Uf(1_Um)+1_Vf(1_Vm) \\ &= (1_Uf)(m)+(1_Vf)(m)  = ((1_U+1_V)f)(m).
\end{align*}

\begin{Prop}\label{p:counit}
If $M$ is a $K\mathcal G$-module, then $\Hom_{K\mathcal G^0}(K\mathcal G^0, M)\cong M$ naturally in $M$.
\end{Prop}
\begin{proof}
    There is a $K\mathcal G^0$-module isomorphism taking $f\colon \mathcal G^0\to M$ to $f(1)$.  Let's show it is a $K\mathcal G$-module isomorphism.   Indeed, if $s\in S$, then $(sf)(1) = sf(s^*1s) = sf(s^*s) = ss^*sf(1) = sf(1)$.  
\end{proof}

We now show that the hom-tensor adjunction holds at the level of $K\mathcal G$-modules.

\begin{Prop}\label{p:hom.tensor}
    Let $L,M,N$ be $K\mathcal G$-modules.   Then there are natural isomorphisms of $K$-modules $\Hom_{K\mathcal G}(L,\Hom_{K\mathcal G^0}(M,N))\cong \Hom_{K\mathcal G}(L\otimes_{K\mathcal G^0} M,N)\cong \Hom_{K\mathcal G}(M,\Hom_{K\mathcal G^0}(L,N))$.
\end{Prop}
\begin{proof}
    Indeed, all three $K$-modules are isomorphic to the $K$-module of $K\mathcal G^0$-bilinear maps $f\colon L\times M\to N$ satisfying $f(s\ell,sm) = sf(\ell,m)$ for $s\in S$.   

    Since this is clear for the middle hom set, and the third hom set is dual to the first, we just handle the first hom set.  If $f\colon L\times M\to N$ satisfies the above property, define $F\colon L\to \Hom_{K\mathcal G}(M,N)$ by $F(\ell)(m) = f(\ell,m)$.  Then $F(s\ell) = f(s\ell, m)= f(ss^*s\ell,m) = f(s\ell,ss^*m) = sf(\ell, s^*m) =sF(\ell)(s^*m) = (sF(\ell))(m)$, where we have used that $f$ is $K\mathcal G^0$-bilinear in the third equality.

    Conversely, if $F\colon L\to \Hom_{K\mathcal G^0}(M,N)$ is a homomorphism and $f(\ell,m)=F(\ell)(m)$, then $f(s\ell,sm) = F(s\ell)(sm) = (sF(\ell))(sm) = sF(\ell)(s^*sm) = ss^*sF(\ell)(m) = sF(\ell)(m) = sf(\ell,m)$.
\end{proof}

Our first corollary of the hom-tensor adjunction is the following.

\begin{Cor}
    If $M,N$ are $K\mathcal G$-modules, then there is a  natural isomorphism $\Hom_{K\mathcal G}(M,N)\cong \Hom_{K\mathcal G}(K\mathcal G^0,\Hom_{K\mathcal G^0}(M,N))$.
\end{Cor}
\begin{proof}
    By Propositions~\ref{p:unit} and~\ref{p:hom.tensor} we have isomorphisms $\Hom_{K\mathcal G}(M,N)\cong \Hom_{K\mathcal G}(K\mathcal G^0\otimes_{K\mathcal G^0}M,N)\cong \Hom_{K\mathcal G}(K\mathcal G^0,\Hom_{K\mathcal G^0}(M,N))$.
\end{proof}

Another consequence of the hom-tensor adjunction is that the tensor product of projective modules is projective when $\mathcal G$ is paracompact (and Hausdorff).  
The following is a slight variation of~\cite[Lemma~2.13]{BDGW}.

\begin{Lemma}\label{l:CDGW}
    Let  $\mathcal G$ be paracompact ample Hausdorff groupoid with compact unit space.  Then $K\mathcal G$ is a projective $K\mathcal G^0$-module.  Consequently, any projective $K\mathcal G$-module is also a projective $K\mathcal G^0$-module.  
\end{Lemma}
\begin{proof}
    Since $\mathcal G$ is paracompact locally compact Hausdorff and totally disconnected space, it has a  covering by pairwise disjoint compact open bisections $\{U_{\alpha}\}_{\alpha \in A}$. Then $K\mathcal G\cong \bigoplus_{\alpha\in A} K\mathcal G^01_{U_\alpha}\cong \bigoplus_{\alpha\in A} K\mathcal G^01_{\ran(U_\alpha)}$ as a $K\mathcal G^0$-module.  
\end{proof}

\begin{Cor}\label{p:tensor.proj}
Let $\mathcal G$ be an ample groupoid with compact unit space.  Let $P,Q$ be $K\mathcal G$-modules with $P$ projective over $K\mathcal G$ and $Q$ projective over $K\mathcal G^0$.  Then $P\otimes_{K\mathcal G^0}Q$ is projective over $K\mathcal G$.      
\end{Cor}
\begin{proof}
    We have by Proposition~\ref{p:hom.tensor} that \[\Hom_{K\mathcal G}(P\otimes_{K\mathcal G^0}Q,-)\cong \Hom_{K\mathcal G}(P,\Hom_{K\mathcal G^0}(Q,-)),\] and the right hand side is exact as $P$ is projective over $K\mathcal G$ and $Q$ is projective over $K\mathcal G^0$.
\end{proof}

Assume now that $\mathcal G$ is a paracompact Hausdorff ample groupoid with compact unit space.
Notice that if $M,N$ are $K\mathcal G$-modules and $P_\bullet\to M$ is a projective resolution over $K\mathcal G$, then it is a projective resolution over $K\mathcal G^0$ by Lemma~\ref{l:CDGW}.  Therefore, $H^n(\Hom_{K\mathcal G^0}(P_\bullet,N))=\Ext^n_{K\mathcal G^0}(M,N)$, and it has a $K\mathcal G$-module structure, which is easily seen to be independent of the projective resolution by a standard argument.  

We now construct a spectral sequence connecting $\mathrm{Ext}$ with groupoid cohomology.    A similar spectral sequence for computing Hochschild cohomology appears in~\cite{dks25}, but they work with the cohomology of the inverse semigroup of compact open bisections.

\begin{Thm}\label{t:ss}
    Let $\mathcal G$ be a paracompact Hausdorff ample groupoid with compact unit space.  Then there is a convergent spectral sequence
    \[\Ext^p_{K\mathcal G}(K\mathcal G^0,\Ext^q_{K\mathcal G^0}(M,N))\Rightarrow \Ext_{K\mathcal G}^{p+q}(M,N)\]
    for $K\mathcal G$-modules $M,N$.
\end{Thm}
\begin{proof}
    Let $P_\bullet\to K\mathcal G^0$ and $Q_\bullet\to M$ be projective resolutions over $K\mathcal G$.  Consider the double cochain complex $C^{p,q} = \Hom_{K\mathcal G}(P_p\otimes_{K\mathcal G^0}Q_q,N)$.   Let us compute the spectral sequence of the first filtration.  We have \[C^{p,\bullet} = \Hom_{K\mathcal G}(P_p\otimes_{K\mathcal G^0} Q_\bullet,N)\cong \Hom_{K\mathcal G}(P_p,\Hom_{K\mathcal G^0}(Q_\bullet,N))\] which has cohomology $\Hom_{K\mathcal G}(P_p,\Ext^q_{K\mathcal G^0}(M,N))$ since $P_p$ is projective for all $p$.  This complex in turn has cohomology $\Ext^p_{K\mathcal G}(K\mathcal G^0,\Ext^q_{K\mathcal G^0}(M,N))$.

    For the second filtration we compute 
    \[C^{\bullet,p}= \Hom_{K\mathcal G}(P_\bullet\otimes_{K\mathcal G^0} Q_p,N)\cong \Hom_{K\mathcal G}(Q_p,\Hom_{K\mathcal G^0}(P_\bullet,N)).\] Since $Q_p$ is projective, the cohomology of this is $\Hom_{K\mathcal G}(Q_p,\Ext^q_{K\mathcal G^0}(K\mathcal G^0,N))$, which then has cohomology $\Ext^p_{K\mathcal G}(M,\Ext^q_{K\mathcal G^0}(K\mathcal G^0,N))$.  But  \[\Ext^q_{K\mathcal G^0}(K\mathcal G^0,N) \cong\begin{cases} N, & \text{if}\ q=0\\ 0, & \text{if}\ q>0 \end{cases}\] by Proposition~\ref{p:counit}.  Thus this spectral sequence collapses on the axis $q=0$, and so we deduce the total complex has homology $\Ext^n_{K\mathcal G}(M,N)$.  This completes the proof of the theorem.
\end{proof}

Next, we need an observation about $KX$ when $X$ is a second countable compact Hausdorff totally disconnected space. Recall that a ring $R$ is (von Neumann) regular~\cite{Goodearlreg} if, for all $a\in R$, there exists $b\in R$ with $aba=a$.

\begin{Thm}\label{t:kaplansky}
Let $X$ be a second countable compact Hausdorff totally disconnected space.  Then $KX$ is hereditary (as a ring with pointwise product).
\end{Thm}
\begin{proof}
    Note first that $KX$ is regular.  Indeed, if $f\in KX$ and one defines $g\in KX$ by \[g(x) = \begin{cases}f(x)\inv, & \text{if}\ f(x)\neq 0\\ 0, & \text{else,}\end{cases}\] then $fgf=f$.   Since $X$ is second countable $KX$ has countable dimension over $K$ (being spanned by the characteristic functions of the countably many compact open sets).  Hence, all ideals in $KX$ are countably generated.  A theorem of Kaplansky says that any countably generated left ideal in a regular ring is projective~\cite[Corollary~2.15]{Goodearlreg}.  We conclude that $KX$ is hereditary. 
\end{proof}

We remark that second countability is needed here.  There are examples of boolean rings of arbtirary global dimension~\cite{PIERCE196791}.  Every boolean ring is of the form $\mathbb F_2X$ for a compact Hausdorff and totally disconnected space $X$, where $\mathbb F_2$ is the two element field.   

We can now prove Theorem~C.
\begin{proof}[Proof of Theorem~C]
    Let $\mathcal G=F_A\ltimes X$.  Note that $\mathcal G=\bigcup_{w\in F_A}\{w\}\times X_{w\inv}$ is a covering of $\mathcal G$ by pairwise disjoint compact open bisections, from which it follows that $\mathcal G$ is paracompact.   Let $L$ be a left ideal of $K\mathcal G$.  We show that $L$ has projective dimension at most $1$; it will then follow $K\mathcal G$ has global dimension at most $2$.  First note that $L$ is projective as a $KX$-module since $KX$ is hereditary by Theorem~\ref{t:kaplansky} and $L$ is a submodule of a projective $KX$-module by Lemma~\ref{l:CDGW}.    Therefore, the spectral sequence in Theorem~\ref{t:ss} (with $M=L$) collapses on the axis $q=0$, from which we conclude that $\Ext^n_{K\mathcal G}(L,N)\cong \Ext^n_{K\mathcal G}(K\mathcal G^0,\Hom_{K\mathcal G^0}(L,N))$.  By Corollary~\ref{c:dim.one} and Lemma~\ref{l:pd.drops}, we have that $\pd_{K\mathcal G} K\mathcal G^0\leq 1$. It follows that $\Ext^2_{K\mathcal G}(L,N)=0$ for all modules $N$.  Thus $L$ has projective dimension at most $1$.
\end{proof}

I conjecture that under the hypothesis of Theorem~C the global dimension is at most one, i.e., $K[F_A\ltimes X]$ is hereditary. 

\begin{Cor}
    Let $(X,T)$ be a Deaconu-Renault system with $X$ second countable compact Hausdorff and totally disconnected and $\dom T$ clopen.  Then $K\mathcal G_{(X,T)}$ has global dimension at most $2$ for any field $K$.
\end{Cor}

The following result about the algebra of a free inverse monoid seems to be new. 

\begin{Cor}
    The algebra of a free inverse monoid over a field has global dimension at most $2$.  
\end{Cor}
\begin{proof}
Let $K$ be a field.    For any inverse monoid $M$, there is an action of $M$  on the spectrum $\wh {E(M)}$ of its idempotents, which is a compact totally disconnected Hausdorff space, such that $KM\cong K[M\ltimes \wh {E(M)}]$~\cite[Theorem~6.3]{mygroupoidalgebra}.  Applying this to the free inverse monoid and using Proposition~\ref{p:fim}, we see that $KFIM_A\cong K[F_A\ltimes \wh{E(FIM_A)}]$ for a semi-saturated partial action of $F_A$ on $\wh{E(FIM_A)}$.   Thus the global dimension of $KFIM_A$ is at most $2$ by Theorem~C.
\end{proof}

I conjecture $KFIM_A$ is hereditary.
 
\bibliographystyle{abbrv}
\bibliography{standard2}
\end{document}